\documentclass[oneside,english]{amsart}
\usepackage[T1]{fontenc}
\usepackage[latin9]{inputenc}
\usepackage{float}
\usepackage{amstext}
\usepackage{amsthm}
\usepackage{amssymb}

\makeatletter

\providecommand{\tabularnewline}{\\}

\numberwithin{equation}{section}
\numberwithin{figure}{section}
\theoremstyle{plain}
\newtheorem{thm}{\protect\theoremname}
  \theoremstyle{definition}
  \newtheorem{defn}[thm]{\protect\definitionname}
  \theoremstyle{plain}
  \newtheorem{prop}[thm]{\protect\propositionname}
  \theoremstyle{plain}
  \newtheorem{cor}[thm]{\protect\corollaryname}
  \theoremstyle{remark}
  \newtheorem{rem}[thm]{\protect\remarkname}
  \theoremstyle{plain}
  \newtheorem{lem}[thm]{\protect\lemmaname}

\makeatother

\usepackage{babel}
  \providecommand{\corollaryname}{Corollary}
  \providecommand{\definitionname}{Definition}
  \providecommand{\lemmaname}{Lemma}
  \providecommand{\propositionname}{Proposition}
  \providecommand{\remarkname}{Remark}
\providecommand{\theoremname}{Theorem}

\begin{document}
\global\long\def\G{\mathcal{G}}
 \global\long\def\F{\mathcal{F}}
 \global\long\def\H{\mathcal{H}}
 \global\long\def\L{\mathcal{L}}
\global\long\def\U{\mathcal{U}}
\global\long\def\W{\mathcal{W}}
 \global\long\def\P{\mathcal{P}}
 \global\long\def\B{\mathcal{B}}
 \global\long\def\A{\mathcal{A}}
\global\long\def\D{\mathcal{D}}
\global\long\def\O{\mathcal{O}}
 \global\long\def\N{\mathcal{N}}
 \global\long\def\X{\mathcal{X}}
 \global\long\def\lm{\mathrm{lim}}
 \global\long\def\then{\Longrightarrow}

\global\long\def\V{\mathcal{V}}
\global\long\def\C{\mathcal{C}}
\global\long\def\adh{\mathrm{\mathrm{adh}}}
\global\long\def\Seq{\mathrm{Seq\,}}
\global\long\def\intr{\mathrm{int}}
\global\long\def\cl{\mathrm{cl}}
\global\long\def\inh{\mathrm{inh}}
\global\long\def\diam{\mathrm{diam}}
\global\long\def\card{\mathrm{card\,}}
\global\long\def\S{\operatorname{S}}
\global\long\def\T{\operatorname{T}}
\global\long\def\I{\operatorname{I}}
\global\long\def\BaseD{\operatorname{B}_{\mathbb{D}}}

\global\long\def\fix{\mathrm{fix\,}}
\global\long\def\Epi{\mathrm{Epi}}

\title[convergence and symmetrizability ]{The convergence-theoretic approach to weakly first countable spaces
and symmetrizable spaces}

\author{Fadoua Chigr and Frédéric Mynard}

\address{New Jersey City University, 2039 J. F. Kennedy Blvd, Jersey City,
NJ 07305}

\email{fmynard@njcu.edu}

\email{fadoua.b@gmail.com}

\thanks{This work was supported by U.S. Education Department Hispanic-Serving
Institutions\textemdash Science, Technology, Engineering \& Math (HSI-STEM)
program grant \# P031C160155, and by the Separately Budgeted Research
program of NJCU for the academic year 2017-2018.}
\maketitle

\section{Introduction}

The convergence-theoretic approach to general topology as developed
by S. Dolecki and his collaborators has reached a level of maturity
allowing for classical general topology to be revisited systematically
in a textbook-like format \cite{DM.book}. After re-interpreting topological
notions in terms of functorial inequalities in the category of convergence
spaces and continuous maps, many classical results are unified and
refined with proofs that are reduced to an algebraic calculus of inequalities.

This article fits in this context and serves as yet another illustration
of the power of the method. More specifically, we spell out convergence-theoretic
characterizations of the notions of weak base, weakly first-countable
space, semi-metrizable space, and symmetrizable spaces. We should
note that these easy characterizations are probably part of the folklore
of convergence space theory, but we couldn't find them explicitly
stated in the literature. With the help of the already established
similar characterizations of the notions of Fréchet-Ursyohn, sequential
\cite{KR.decompositionseries}\cite{quest2}, and accessibility spaces
\cite{dolecki1998topologically}, we give a simple algebraic proof
of a classical result regarding when a symmetrizable space is semi-metrizable,
a weakly first-countable space is first countable, and a sequential
space is Fréchet, that clarifies the situation for non-Hausdorff spaces. 

On the other hand, these reformulations bring the notions of symmetrizability
and weak first-countability into the fold of the general program studying
stability of of topological properties under product with the tools
of \emph{coreflectively modified duality} \cite{DM.products,mynard,mynard.strong,Mynard.survey,coreflrevisited}.
This way, we obtain simple algebraic proofs of various results (of
Y. Tanaka) on the stability under product of symmetrizability and
weak first-countability, and we obtain the same way a (as far as we
know, entirely new) characterization of spaces whose product with
every metrizable topology is weakly first-countable, respectively
symmetrizable. 

We use notations and terminology consistent with the recent book \cite{DM.book},
to which we refer the reader for undefined notions, and for a comprehensive
treatment of convergence spaces. We gather notation, conventions,
and essential notions of convergence theory we refer to in Appendi
at the end of this article, which the reader unfamiliar with convergence
spaces should probably read first.

\section{Weak base and weakly first-countable spaces}

Recall the classical topological notion of weak base.
\begin{defn}
Let $X$ be a topological space. For each $x\in X$, let $\B_{x}$
be a filter-base on $X$ such that $x\in\bigcap\B_{x}$. The family
\[
\B=\bigcup_{x\in X}\B_{x}
\]
 is called a \emph{weak base} for $X$ if $U\subset X$ is open if
and only if for every $x\in U$ there is $V\in\B_{x}$ with $x\in V\subset U$.

A space $X$ is \emph{weakly first-countable }if it has a weak base
such that $\B_{x}$ is countable, for each $x\in X$. 
\end{defn}
The notion of weakly first-countable space was introduced by Arhangel'skii
in \cite{arh.mapppings}\emph{ }and renamed \emph{g-first countable
}by Siwiec in \cite{MR0350706}\emph{. }In view of the following proposition,
we will call it \emph{of countable pretopological character}. Let
$\S_{0}$ and $\T$ denote the reflectors from the category of convergence
spaces onto that of pretopological and topological spaces respectively.
See the Appendix for definitions of notions of convergence theory.

\begin{prop}
\label{prop:weakbase} 
\begin{enumerate}
\item If $\xi$ is a pretopology and $\B_{x}$ is a filter-base of $\V_{\xi}(x)$
for every $x\in X$, then $\B=\bigcup_{x\in X}\B_{x}$ is a weak base
of $\T\xi$.
\item If $\B=\bigcup_{x\in X}\B_{x}$ is a weak base of a topology $\tau$
on $X$, then a pretopology $\xi$ defined by $\V_{\xi}(x):=\B_{x}^{\uparrow}$
for every $x\in X$ fulfills $\tau=\T\xi$.
\end{enumerate}
\end{prop}
\begin{proof}
(1): Assume $\xi=\operatorname{S}_{0}\xi$ with $\T\xi=\tau$ and
$\B_{x}$ is a filter-base of $\V_{\xi}(x)$ for each $x\in X$. Let
$U\in\O_{\tau}$ with $x\in U$. Because $\xi\geq\tau$, $\V_{\xi}(x)\geq\N_{\tau}(x)$
so that there is $V\in\B_{x}$ with $x\in V\subset U$. If conversely,
$U$ is such that for every $x\in U$ there is $V\in\V_{\xi}(x)$
with $x\in V\subset U$ then $U\subset\inh_{\xi}U$ and thus $U$
is $\xi$-open, hence $\tau$-open.

(2): if $\B=\bigcup_{x\in X}\B_{x}$ is a weak base for $\tau$ and
$\xi$ is a pretopology defined by $\V_{\xi}(x):=\B_{x}^{\uparrow}$
, then $\T\xi=\tau$. Indeed, $U$ is $\xi$-open if and only if $U\in\bigcap_{x\in U}\V_{\xi}(x)$,
which is here the weak base condition.
\end{proof}
In other words,
\begin{prop}
A weak base of a topology $\tau$ is a (convergence) base of a pretopology
$\xi$ for which $\T\xi=\tau.$
\end{prop}
Therefore, we alternatively call a weak base a \emph{pretopological
base }or $\S_{0}$-\emph{base}, and accordingly, we say that a topology
is \emph{of countable pretopological character }or \emph{of countable
}$\S_{0}$-\emph{character }rather than use the terms \emph{weakly
first-countable }or \emph{g-first-countable }introduced before. We
call a topology with a countable local base (and by extension, convergences
based in countably-based filters) \emph{of countable character }rather
than \emph{first-countable.}

In view of Proposition \ref{prop:weakbase}, 
\begin{cor}
\label{cor:gfirst} A topology $\tau$ has countable $\S_{0}$-character
if and only if there exists a pretopology $\xi$ of countable character
with $\T\xi=\tau$.
\end{cor}
Let $\I_{1}$ denote the the coreflector on convergence spaces of
countable character. The sequentiality of a topology $\tau$ can be
characterized by the inequality $\tau\geq\T\I_{1}\tau$ (See Appendix
for details). Combining this fact with Corollary \ref{cor:gfirst},
it is clear that each space of countable $\S_{0}$-character is sequential.

Let us call a map $d:X\times X\to\mathbb{R}^{+}$ such that $d(x,x)=0$
for all $x\in X$ a \emph{protometric }on $X$. The balls $B(x,\epsilon)=\{y\in X:d(x,y)<\epsilon\}$
around a point $x$ for a protometric form a filter-base with $x$
in its intersection. Hence a protometric $d$ defines a pretopology
$\tilde{d}$ of countable character given by the vicinity filter
\begin{equation}
\V_{\tilde{d}}(x)=\left\{ B(x,\epsilon):\epsilon>0\right\} ^{\uparrow}=\Big\{ B\Big(x,\frac{1}{n}\Big):n\in\mathbb{N}\Big\}^{\uparrow},\label{eq:dtilde}
\end{equation}
at each $x\in X$. We will say that a convergence $\xi$ is \emph{protometrizable
}if there is a protometric $d$ on $|\xi|$ with $\xi=\tilde{d}$. 
\begin{prop}
\label{prop:protometrizable} A convergence is protometrizable if
and only if it is a pretopology of countable character.
\end{prop}
\begin{proof}
Let $\sigma$ be a pretopology of countable character. For each $x\in|\sigma|$,
let $\{V_{n}(x):n\in\mathbb{N}\}$ be a decreasing filter-base of
$\V_{\sigma}(x)$. Let $d:|\sigma|\times|\sigma|\to[0,\infty)$ be
defined by $d(x,y)=0$ if $y\in\bigcap_{n\in\mathbb{N}}V_{n}(x)$,
$d(x,y)=\frac{1}{n}$ if $y\in V_{n}(x)\setminus V_{n+1}(x)$, and
$d(x,y)=1$ otherwise. By definition, $d$ is a protometric because
$x\in\bigcap_{n\in\mathbb{N}}V_{n}(x)$. Moreover, $d(x,y)\leq\frac{1}{n}$
if and only if $y\in V_{n}(x)$ so that $\V_{\tilde{d}}(x)=\V_{\sigma}(x)$,
that is, $\sigma=\tilde{d}$.
\end{proof}
Hence \emph{every }(pre)topology of countable character is protometrizable.
Moreover, 
\[
x\in\lm_{\tilde{d}}\{y\}^{\uparrow}\iff d(x,y)=0,
\]
so that there is no hope of protometrizing every topology of countable
character with a \emph{symmetric} protometric (\footnote{that is, a protometric $d$ on $X$ satisfying $d(x,y)=d(y,x)$ for
all $x,y\in X$}) for there are topological spaces (say, the Sierpi\'nski space)
with points $x$ and $y$ with $x\in\lim\{y\}^{\uparrow}$ but $y\notin\lim\{x\}^{\uparrow}$.

Nedev called a topological space $(X,\tau)$ \emph{o-metrizable }\cite{MR0367935}
if there is a protometric $d$ on $X$ for which $O\subset X$ is
open if and only if for every $x\in O$, there is $\epsilon>0$ with
$B(x,\epsilon)\subset O$. In our terms, this means that $\tau$ is
o-metrizable if there is a protometric $d$ on $X$ with $\T\tilde{d}=\tau$
(\footnote{For $O\in\O_{\tau}=\O_{\tilde{d}}$ if and only if $O\in\bigcap_{x\in O}\V_{\tilde{d}}(x)$,
that is, if and only if for every $x\in O$, there is $\epsilon_{x}$
with $x\in B(x,\epsilon_{x})\subset O$. }). Therefore, Corollary \ref{cor:gfirst} and Proposition \ref{prop:protometrizable}
combine to the effect that
\begin{prop}
A topology is o-metrizable if and only if it has countable $\S_{0}$-character.
\end{prop}
More generally, we say that a convergence $\xi$ \emph{has countable
$\S_{0}$-character }or is $\S_{0}$-\emph{characterized} if there
is a pretopology $\sigma$ of countable character with 
\[
\sigma\geq\xi\geq\T\sigma.
\]

In the Hausdorff case, there is a unique choice for the pretopology
$\sigma$:
\begin{prop}
\label{prop:HausdorffI1xi} A Hausdorff convergence $\xi$ has countable
$\S_{0}$-character if and only if $\I_{1}\xi$ is a pretopology and
$\xi$ is sequential.

In this case, $\I_{1}\xi$ is the only pretopology $\sigma$ of countable
character with $\sigma\geq\xi\geq\T\sigma$.
\end{prop}
\begin{proof}
If $\xi$ has countable $\S_{0}$-character then $\xi$ is in particular
sequential, and there is $\sigma=\S_{0}\sigma=\I_{1}\sigma$ with
$\sigma\geq\xi\geq\T\sigma$. In particular, $\sigma\geq\I_{1}\xi$.
Moreover, to see the reverse inequality, note that if $x\in\lim_{\I_{1}\xi}\F$
then there is $\H\in\mathbb{F}_{1}$ with $\H\leq\F$ and $x\in\lim_{\xi}\H$.
For every $\mathcal{E}\in\mathbb{E}|\xi|$ with $\mathcal{E\geq\H},$
\[
x\in\lm_{\Seq\xi}\mathcal{E}\subset\lm_{\Seq\T\sigma}\mathcal{E}=\lm_{\Seq\sigma}\mathcal{E}
\]
by Proposition \ref{prop:coincideSeq}, so that $x\in\lm_{\sigma}\mathcal{E}$.
As $\sigma=\S_{0}\sigma$, 
\[
x\in\lm_{\sigma}\bigwedge_{\mathcal{E}\in\mathbb{E}|\xi|,\mathcal{E\geq\H}}\mathcal{E},
\]
and $\bigwedge_{\mathcal{E}\in\mathbb{E}|\xi|,\mathcal{E\geq\H}}\mathcal{E}=\H$
because $\H\in\mathbb{F}_{1}$, hence is a Fréchet filter, so that
$x\in\lm_{\sigma}\H\subset\lim_{\I_{1}\xi}\F$. Thus $\sigma=\I_{1}\xi$
is pretopological and is the only $\sigma$ witnessing the definition
of countable $\S_{0}$-character. 

Conversely, if $\I_{1}\xi$ is a pretopology and $\xi$ is sequential,
that is, $\xi\geq\T\I_{1}\xi$, then $\sigma=\I_{1}\xi$ witnesses
the definition of countable $\S_{0}$-character of $\xi$.
\end{proof}
As a result, if $\xi$ is a Fréchet Hausdorff pretopology of countable
$\S_{0}$-character, then $\xi=\S_{0}\I_{1}\xi$ and $\I_{1}\xi$
is pretopological, so that $\xi=\I_{1}\xi$ has countable character.
This is a classical result that we will refine in the Section \ref{sec:Semimet}.

\section{Semi-metrizable spaces and symmetrizable spaces\label{sec:Semimet}}

We call \emph{semi-metric} a symmetric protometric $d$ on $X$ such
that $d(x,y)=0$ if and only if $x=y$. This is also called \emph{premetric}
in \cite{DM.book}. Traditionally, e.g., \cite{Gruenhage84}, a topological
space $(X,\tau)$ is called \emph{semi-metrizable }if there is a semi-metric
$d$ on $X$ for which, for every $x\in X$ and $A\subset X$,
\begin{equation}
x\in\cl_{\tau}A\iff d(x,A):=\inf_{a\in A}d(x,a)=0,\label{eq:semimetricclosure}
\end{equation}
and \emph{symmetrizable }if there is a semi-metric $d$ on $X$ for
which $O$ is $\tau$-open if and only if for every $x\in O$, there
is $\epsilon>0$ with $B(x,\epsilon)\subset O$. This can easily be
reformulated in terms more directly comparable to (\ref{eq:semimetricclosure}):
$\tau$ is symmetrizable if and only if $F\subset X$ is $\tau$-closed
if and only if $d(x,F)>0$ for every $x\in X\setminus F$. Hence a
semi-metrizable space is symmetrizable. Reformulating the definitions
in terms of the induced pretopology, we obtain:
\begin{thm}
\label{thm:semimet} Let $(X,\tau)$ be a topological space. 
\begin{enumerate}
\item $(X,\tau)$ is semi-metrizable if and only if there is a semi-metric
$d$ on $X$ for which $\tau=\tilde{d}$;
\item $(X,\tau)$ is symmetrizable if and only if there is a semi-metric
$d$ on $X$ for which $\tau=\T\tilde{d}$.
\end{enumerate}
\end{thm}
\begin{proof}
(1) If $\tau$ is semi-metrizable there is a semi-metric $d$ satisfying
(\ref{eq:semimetricclosure}). Then $\tilde{d}\geq\tau$ because $\V_{\tilde{d}}(x)\geq\N_{\tau}(x)$.
Indeed, if $O$ is a $\tau$-open set and $O\notin\V_{\tilde{d}}(x)$,
then $X\setminus O\in(\V_{\tilde{d}}(x))^{\#}$ so that $d(x,X\setminus O)=0$,
hence $x\in\cl_{\tau}(X\setminus O)=X\setminus O$ by (\ref{eq:semimetricclosure}),
that is, $x\notin O$. Conversely, $\tau\geq\tilde{d}$ because $\N_{\tau}(x)\geq\V_{\tilde{d}}(x)$.
Indeed, if there is $\epsilon>0$ with $B(x,\epsilon)\notin\N_{\tau}(x)$
then $X\setminus B(x,\epsilon)\in(\N_{\tau}(x))^{\#}$ so that $x\in\cl_{\tau}(X\setminus B(x,\epsilon))$.
By (\ref{eq:semimetricclosure}), $d(x,X\setminus B(x,\epsilon))=0$
which is not possible, for $d(x,y)\geq\epsilon$ for every $y\notin B(x,\epsilon)$.

(2) The topology $\tau$ is symmetrizable by a semi-metric $d$ if
and only if $O$ is $\tau$-open exactly if for every $x\in O$, there
is $\epsilon>0$ with $B(x,\epsilon)\subset O$, that is, if and only
if $O\in\bigcap_{x\in O}\V_{\tilde{d}}(x)$, equivalently, if $O$
is $\tilde{d}$-open. In other words, $\tau$ is symmetrizable by
$d$ if and only if $\tau$-open sets and $\tilde{d}$-open sets coincide,
that is, if and only if $\tau=\T\tilde{d}.$
\end{proof}
\begin{rem}
A characterization of semi-metrizable (resp. symmetrizable) topologies
by R. Heath in \cite{heath1962arc} is that a topological space $X$
is semi-metrizable if and only if it is $T_{1}$ and for every $x\in X$,
there is a decreasing countable local (resp. pretopological) base
$\{U_{n}(x)\}_{n=1}^{\infty}$ such that $x\in\lm(x_{n})_{n}$ whenever
$x\in U_{n}(x_{n})$ for $n\in\mathbb{N}$.
\end{rem}
Note that we can now more generally define a convergence $\xi$ to
be \emph{semi-metrizable }if $\xi=\tilde{d}$ for some semi-metric
$d$ on $|\xi|$ (which of course imposes that $\xi$ be a pretopology)
and \emph{symmetrizable }if 
\[
\tilde{d}\geq\xi\geq\T\tilde{d}
\]
for some semi-metric $d$ on $|\xi|$.

In view of Theorem \ref{thm:semimet}, Corollary \ref{cor:gfirst}
and the characterizations (\ref{eq:Frechet}) and (\ref{eq:sequential})
of Fréchet and sequential spaces, the following immediate relations
between the notions at hand hold:
\begin{center}
\begin{table}[H]
\begin{centering}
\begin{tabular}{ccc}
Fréchet &  & Sequential\tabularnewline
$\xi\geq\S_{0}\I_{1}\xi$ & $\Longrightarrow$ & $\xi\geq\T\I_{1}\xi$\tabularnewline
$\Uparrow$ &  & $\Uparrow$\tabularnewline
pretop. of countable character &  & of countable $\S_{0}$-character\tabularnewline
$\xi=\S_{0}\xi=\I_{1}\xi$ & $\Longrightarrow$ & $\exists\sigma=\S_{0}\sigma=\I_{1}\sigma:\sigma\geq\xi\geq\T\sigma$\tabularnewline
$\Uparrow$ &  & $\Uparrow$\tabularnewline
semi-metrizable &  & symmetrizable\tabularnewline
$\exists d:\xi=\tilde{d}$ & $\Longrightarrow$ & $\exists d:\tilde{d}\geq\xi\geq\T\tilde{d}$\tabularnewline
\end{tabular}
\par\end{centering}
\bigskip{}

\caption{None of the implications can be reversed in general: The topological
modification of the Féron cross pretopology (e.g, \cite[Example I.2.1]{DM.book})
is symmetrizable but not Fréchet, showing that none of the left to
right implications can be reversed. The sequential fan is a Fréchet
topology that is not of countable $\protect\S_{0}$-character, and
the Sierpi\'nski space is of countable character but is not symmetrizable,
showing that none of the vertical implication can be reversed.}
\label{table1}
\end{table}
\par\end{center}

Analogously to Proposition \ref{prop:HausdorffI1xi}, we have:
\begin{cor}
\label{cor:symtosemi} If $\xi$ is a Hausdorff convergence, then
it is symmetrizable if and only if $\xi$ is sequential and $\I_{1}\xi$
is a semi-metrizable pretopology.
\end{cor}
\begin{proof}
If $\xi$ is symmetrizable, it is in particular of countable $\S_{0}$-character,
so that $\sigma=\I_{1}\xi$ is the only pretopology of countable character
with $\sigma\geq\xi\geq\T\sigma$. Since there is a semi-metric $d$
with $\tilde{d}\geq\xi\geq\T\tilde{d}$, we conclude that $\tilde{d}=\I_{1}\xi$.

Conversely, if $\xi$ is sequential, that is, $\xi\geq\T\I_{1}\xi$
and $\I_{1}\xi=\tilde{d}$ for some semi-metric $d$, then $\xi$
is symmetrizable.
\end{proof}
Of course, if $\xi$ is Fréchet Hausdorff and symmetrizable then $\xi\geq\S_{0}\I_{1}\xi=\I_{1}\xi=\tilde{d}$
is semi-metrizable.

In other words, the horizontal implications in the diagram above can
be reversed when the space is Hausdorff and Fréchet. This is well
known, and easy to see as we already have proved. The observation
can be traced all the way back to \cite{MR0350706}, it was included
without proof in \cite{harley1976symmetrizable}, and a proof can
be found in the popular survey \cite{Gruenhage84}. Yet approaches
to its proof have remained the subject of more research, e.g., \cite{hong1999notes,hong2002note}
without yielding any refinement. The convergence-theoretic viewpoint
provides as a consequence such a slight refinement by clarifying what
condition is needed if separation is dropped, with an immediate and
transparent proof. The relevant notion is that of \emph{accessibility
space }as introduced by Whyburn in \cite{why}.

\section{Accessibility spaces}

Recall from \cite{why} that a topological space $(X,\tau)$ is \emph{of
accessibility, }or is an \emph{accessibility space,} if for each $x_{0}\in X$
and every $H\subset X$ with $x_{0}\in\cl_{\tau}(H\setminus\{x_{0}\}),$
there is a closed subset $F$ of $X$ with $x_{0}\in\cl(F\setminus\{x_{0}\})$
and $x_{0}\notin\cl_{\tau}(F\setminus H\setminus\{x_{0}\})$. Note
that if $x_{0}\in\cl_{\tau}(H\setminus\{x_{0}\})$ in a Fréchet topological
space, there is a sequence $\{x_{n}\}_{n}$ on $H\setminus\{x_{0}\}$
converging to $x_{0}$. The set $F=\{x_{n}:n\in\mathbb{N}\}\cup\{x_{0}\}$
then witnesses the definition of accessibility, provided that $x_{0}$
is the unique limit point of $\{x_{n}\}_{n}$. Hence, 
\begin{prop}
\label{prop:FHisaccess} A Fréchet topological space in which sequences
have unique limits, in particular a Fréchet Hausdorff space, is of
accessibility.
\end{prop}
Following \cite{DG}, we say that a pretopology $\tau$ is \emph{topologically
maximal within the class of pretopologies }if 
\[
\sigma=\S_{0}\sigma\geq\tau\text{ and }\T\sigma=\T\tau\then\sigma=\tau.
\]

\begin{thm}
\label{thm:accessiblitiy} \cite{DG}\cite{dolecki1998topologically}
A topology is accessibility if and only if it is topologically maximal
within the class of pretopologies. 
\end{thm}
In Table \ref{table1} above, note that $\xi=\T\xi$ is a topology,
and, considering in each row bottom to top, $\tilde{d}$, $\sigma$
and $\S_{0}\I_{1}\xi$ are pretopologies finer than $\xi$ with $\xi=\T\tilde{d}$,
$\xi=\T\sigma$ and $\xi=\T(\S_{0}\I_{1}\xi)$ respectively. Thus,
it is an immediate consequence of Theorem \ref{thm:accessiblitiy}
that if $\xi$ is additionally an accessibility space, then $\xi=\tilde{d}$,
$\xi=\sigma$ and $\xi=\S_{0}\I_{1}\xi$ respectively, by topological
maximality: In Table \ref{table1}, the two columns coincide if the
space is an accessibility space.
\begin{cor}
\label{thm:access} An accessibility sequential space is Fréchet;
an accessibility space of countable $\S_{0}$-character has countable
character; an accessibility symmetrizable space is semi-metrizable.
\end{cor}
In view of Proposition \ref{prop:FHisaccess}, this refines the classical
result that a Fréchet Hausdorff space of countable $\S_{0}$-character
(respectively symmetrizable) has countable character (respectively,
is semi-metrizable) by relaxing the condition Fréchet Hausdorff to
accessibility.

Note also that Proposition \ref{prop:HausdorffI1xi} and Corollary
\ref{cor:symtosemi} delineate the circumstances under which the vertical
implications can be reversed among Hausdorff convergences: between
the first two rows exactly when $\I_{1}\xi$ is a pretopology, between
the lower two rows exactly when $\I_{1}\xi$ is moreover semi-metrizable.

\section{Strongly $\protect\S_{0}$-characterized and strongly symmetrizable
spaces}

Recall that a topological space is \emph{strongly Fréchet }if whenever
$x\in\adh\H$ for a countably based filter $\H$, then there exists
a sequential filter $(x_{n})_{n}$ with $(x_{n})_{n}\geq\H$ and $x\in\lim(x_{n})_{n}$.
Observe that a topology $\tau$ is strongly Fréchet if and only if
$\tau\geq\S_{1}\I_{1}\tau$ where $\S_{1}$ is the paratopologizer.
It is well-known (e.g., \cite[Prop. 4.D.4 and 4.D.5]{quest}) that
\begin{prop}
The following are equivalent for a topological space $X$:
\begin{enumerate}
\item $X$ is strongly Fréchet;
\item $X\times Y$ is strongly Fréchet for every bisequential space $Y$;
\item $X\times Y$ is Fréchet for every prime metrizable topological space
$Y$.
\end{enumerate}
\end{prop}
\emph{Strongly sequential }spaces were introduced by analogy, and
characterized internally (\footnote{Namely, a $T_{1}$ topology is strongly sequential if 
\[
\lm\F=\bigcap_{\mathbb{F}_{1}\ni\H\#\F}\cl_{\Seq\xi}(\adh_{\Seq\xi}\H)
\]
for every filter $\F$, to be compared with the fact that a topology
is strongly Fréchet if 
\[
\lm\F=\bigcap_{\mathbb{F}_{1}\ni\H\#\F}\adh_{\Seq\xi}\H
\]
for every filter $\F$.}) in \cite{mynard.strong}. They satisfy
\begin{prop}
The following are equivalent for a topological space $X$:
\begin{enumerate}
\item $X$ is strongly sequential;
\item $X\times Y$ is strongly sequential for every bisequential space $Y$;
\item $X\times Y$ is sequential for every prime metrizable topological
space $Y$.
\end{enumerate}
\end{prop}
Accordingly, if $P$ is a topological property, then we say that $X$
is \emph{strongly $P$ }if $X\times Y$ is $P$ for every metrizable
space $Y$.

A product of a topology of countable $\S_{0}$-character with a metrizable
topology need not  have countable $\S_{0}$-character, and a product
of a symmetrizable topology with a metrizable topology need not be
symmetrizable, as shows \cite[Example 4.5]{Tanakonsym}, which is
an example of a symmetrizable (hence of countable $\S_{0}$-character)
space and a metrizable space whose product is not even sequential.

As far as we know, the spaces whose product with every metrizable
topology have countable $\S_{0}$-character, and those whose product
with every metrizable topology is symmetrizable, have not been characterized.
We set out to do this in this section. To apply our convention regarding
``strongly $P$'', it will be convenient to say a convergence is\emph{
countably characterized }if it has countable character and \emph{countably
}$\S_{0}$-\emph{characterized }if it has countable $\S_{0}$-character. 

Recall that the reflector $\Epi_{\I_{1}}$ associates to a convergence
$\xi$ the coarsest among convergences $\theta$ on $|\xi|$ satisfying
\[
\forall\tau=\I_{1}\tau,\;\T(\xi\times\tau)=\T(\theta\times\tau).
\]

We invite the reader to consult \cite{mynard}, \cite{mynard.strong}
or \cite{DM.book} for various characterizations and applications
of this functor. For our purpose, it is enough to know \cite{mynard.strong}
that $\xi\geq\Epi_{\I_{1}}\xi\geq\T\xi$ for every $\xi$ and that
$\Epi_{\I_{1}}\xi$ is characterized by
\begin{eqnarray}
\theta\geq\Epi_{\I_{1}}\xi & \iff & \forall\tau=\I_{1}\tau,\;\theta\times\tau\geq\Epi_{\I_{1}}(\xi\times\tau)\label{eq:Epi}\\
 & \iff & \forall\tau\text{ prime metrizable topology},\;\theta\times\tau\geq\T(\xi\times\tau).\nonumber 
\end{eqnarray}

Strongly sequential spaces were characterized \cite{mynard.strong}
as those topologies $\tau$ satisfying
\begin{equation}
\tau\geq\Epi_{\I_{1}}\I_{1}\tau=\Epi_{\I_{1}}\Seq\tau.\label{eq:strnglyseq}
\end{equation}
\emph{ }We note that:
\begin{prop}
\cite[Proposition 4]{Myn.strongmore} \label{prop:fromstrongmore}
A regular sequential locally countably compact topology is strongly
sequential.
\end{prop}
A convergence $\xi$ is called \emph{strongly countably} $\S_{0}$\emph{-characterized
}if there is a pretopology $\sigma$ of countable character with 
\[
\sigma\geq\xi\geq\Epi_{\I_{1}}\sigma.
\]

\begin{prop}
\label{prop:stronggfirst} A Hausdorff convergence $\xi$ is strongly
countably $\S_{0}$-characterized if and only if it is countably $\S_{0}$-characterized
and strongly sequential.
\end{prop}
\begin{proof}
If $\xi$ is strongly countably $\S_{0}$-characterized then there
is a pretopology $\sigma$ of countable character with $\xi\geq\Epi_{I_{1}}\sigma\geq\T\sigma$
and thus $\xi$ is countably $\S_{0}$-characterized. Since $\xi$
is also Hausdorff, $\sigma=\I_{1}\xi$ by Proposition \ref{prop:HausdorffI1xi},
so that $\xi\geq\Epi_{\I_{1}}\I_{1}\xi$ and $\xi$ is strongly sequential.

Conversely, if $\xi$ is countably $\S_{0}$-characterized and Hausdorff,
then $\I_{1}\xi$ is a pretopology by Proposition \ref{prop:HausdorffI1xi}.
If moreover $\xi\geq\Epi_{I_{1}}\I_{1}\xi$, then $\sigma=\I_{1}\xi$
witnesses the definition of $\xi$ being strongly countably $\S_{0}$-characterized.
\end{proof}
\begin{thm}
\label{thm:stronglygfirst} The following are equivalent for a Hausdorff
convergence $\xi$:
\begin{enumerate}
\item $\xi$ is strongly countably $\S_{0}$-characterized;
\item $\xi\times\tau$ is strongly countably $\S_{0}$-characterized for
every countably characterized pretopology $\tau$;
\item $\xi\times\tau$ is countably $\S_{0}$-characterized for every prime
metrizable topology $\tau$.
\end{enumerate}
\end{thm}
\begin{proof}
$(1)\then(2)$: If $\xi\geq\Epi_{I_{1}}\sigma$ where $\sigma=\I_{1}\sigma=\S_{0}\sigma\geq\xi$
and $\tau=\I_{1}\tau=\S_{0}\tau$ then
\[
\xi\times\tau\geq\Epi_{\I_{1}}(\sigma\times\tau)
\]
and $\sigma\times\tau$ is a pretopology of countable character finer
that $\xi\times\tau$, so that $\xi\times\tau$ is strongly countably
$\S_{0}$-characterized.

$(2)\then(3)$ is clear.

$(3)\then(1)$: If conversely, $\xi\times\tau$ is countably $\S_{0}$-characterized
for every prime topology $\tau$ of countable character, then in particular,
taking a singleton for $\tau$, $\xi$ must be countably $\S_{0}$-characterized.
As $\xi$ is Hausdorff and prime topologies are Hausdorff, we conclude
that $\xi\times\tau$ is Hausdorff and countably $\S_{0}$-characterized
whenever $\tau$ is a prime metrizable topology. In view of Proposition
\ref{prop:HausdorffI1xi}, 
\[
\I_{1}(\xi\times\tau)=\I_{1}\xi\times\tau
\]
is the only pretopology of countable character witnessing the definition
of $\xi\times\tau$ being countably $\S_{0}$-characterized. Hence
\[
\xi\times\tau\geq\T(\I_{1}\xi\times\tau)
\]
for every prime metrizable topology $\tau$ and thus $\xi\geq\Epi_{\I_{1}}\I_{1}\xi$
by (\ref{eq:Epi}). Hence $\xi$ is strongly countably $\S_{0}$-characterized.
\end{proof}
Similarly, we call \emph{strongly symmetrizable }a convergence $\xi$
for which there is a semi-metric $d$ on $|\xi|$ with 
\[
\tilde{d}\geq\xi\geq\Epi_{\I_{1}}\tilde{d},
\]
and we obtain with similar arguments that
\begin{prop}
A Hausdorff convergence $\xi$ is strongly symmetrizable if and only
if it is strongly sequential and symmetrizable.
\end{prop}
\begin{thm}
\label{thm:stronglygfirst-1} The following are equivalent for a Hausdorff
convergence $\xi$:
\begin{enumerate}
\item $\xi$ is strongly symmetrizable;
\item $\xi\times\tau$ is strongly symmetrizable for every semi-metrizable
pretopology $\tau$;
\item $\xi\times\tau$ is symmetrizable for every prime metrizable topology
$\tau$.
\end{enumerate}
\end{thm}
In the diagram below, we omit ``Fréchet'', for a Fréchet space does
not need to be strongly sequential.
\begin{center}
\begin{table}[H]
\begin{centering}
\begin{tabular}{ccccc}
strongly Fréchet &  & strongly sequential &  & Sequential\tabularnewline
$\xi\geq\S_{1}\I_{1}\xi$ & $\Longrightarrow$ & $\xi\geq\Epi_{\I_{1}}\I_{1}\xi$ & $\Longrightarrow$ & $\xi\geq\T\I_{1}\xi$\tabularnewline
$\Uparrow$ &  & $\Uparrow$ &  & $\Uparrow$\tabularnewline
countably  &  & strongly countably &  & countably \tabularnewline
characterized &  &  $\S_{0}$-characterized &  & $\S_{0}$-characterized\tabularnewline
$\xi=\I_{1}\xi$ & $\Longrightarrow$ & $\exists\sigma:\sigma\geq\xi\geq\Epi_{\I_{1}}\sigma$ & $\Longrightarrow$ & $\exists\sigma:\sigma\geq\xi\geq\T\sigma$\tabularnewline
$\Uparrow$ &  & $\Uparrow$ &  & $\Uparrow$\tabularnewline
semi-metrizable &  & strongly symmetrizable &  & symmetrizable\tabularnewline
$\exists d:\xi=\tilde{d}$ & $\Longrightarrow$ & $\exists d:\tilde{d}\geq\xi\geq\Epi_{\I_{1}}\tilde{d}$ & $\Longrightarrow$ & $\exists d:\tilde{d}\geq\xi\geq\T\tilde{d}$\tabularnewline
\end{tabular}
\par\end{centering}
\bigskip{}

\caption{In this table $d$ denotes a semi-metric and $\sigma$ denotes a pretopology
of countable character}
\end{table}
\par\end{center}

\section{More product theorems}

Corollary \ref{cor:gfirst} and Theorem \ref{thm:semimet} also shed
light on results of Y. Tanaka on stability under product of $\S_{0}$-characterized
and symmetrizable spaces:
\begin{thm}
\cite{Tanakonsym}\label{thm:Tanakaprod} Let $X$ be a regular Hausdorff
topological space.
\begin{enumerate}
\item If $X$ is a locally countably compact symmetrizable space then $X\times Y$
is symmetrizable for every symmetrizable space $Y$.
\item If $X$ is locally countably compact and of countable $\S_{0}$-character,
then $X\times Y$ has countable $\S_{0}$-character for every topological
space $Y$ of countable $\S_{0}$-character.
\end{enumerate}
\end{thm}
Both results are, in fact, instances of results of \cite{DM.products}
characterizing the relation between two convergences $\xi$ and $\sigma$
on the same underlying set for which 
\begin{equation}
\xi\times\T\tau\geq\T(\sigma\times\tau).\label{eq:commute}
\end{equation}
for every $\tau$ in a specified class of convergence. The full scope
of results studying (\ref{eq:commute}) is not needed here and we
invite the interested reader to consult \cite{DM.products} for an
account of the many topological applications of variants of (\ref{eq:commute}).
For our present purpose, we will only need this very particular case: 
\begin{thm}
\label{thm:frommech} \cite[Theorem 9.10]{DM.products} Let $\sigma$
be a convergence such that $\T\sigma$ is a regular Hausdorff topology.
The following are equivalent:
\begin{enumerate}
\item $\T\sigma$ is locally countably $\Epi_{\I_{1}}\sigma$-compact;
\item $\T\sigma\times\T\tau\geq\T(\sigma\times\tau)$ for every $\tau\geq\T\I_{1}\tau$;
\item $\T\sigma\times\T\tau\geq\T(\sigma\times\tau)$ for every $\tau=\I_{1}\T\tau$.
\end{enumerate}
\end{thm}
We will need the following technical lemma:
\begin{lem}
\label{lem:loccompactrel} Let $\xi$ be a regular Hausdorff topology
of countable $\S_{0}$-character, so that $\xi=\T\sigma$ for some
pretopology $\sigma$ of countable character. If $\xi$ is locally
countably compact, then $\xi=\Epi_{\I_{1}}\sigma$, and thus $\xi$
is locally countably $\Epi_{\I_{1}}\sigma$-compact. 
\end{lem}
\begin{proof}
Because 
\[
\xi=\T\I_{1}\sigma\geq\T\I_{1}\T\sigma=\T\I_{1}\xi,
\]
 $\xi$ is in particular sequential and regular. If $\xi$ is also
locally countably compact, then by Proposition \ref{prop:fromstrongmore}
it is also strongly sequential so that in view of Proposition \ref{prop:stronggfirst},
$\xi$ is strongly countably $\S_{0}$-characterized, and thus $\xi=\Epi_{\I_{1}}\alpha$
for some pretopology $\alpha$ of countable character. By Proposition
\ref{prop:HausdorffI1xi}, $\alpha=\I_{1}\xi=\sigma$, which concludes
the proof.
\end{proof}
\begin{proof}[Proof of Theorem \ref{thm:Tanakaprod} ]
We prove (2) and (1) follows the same argument replacing the role
of $\sigma$ and $\tau$ by pretopologies of the form $\tilde{d}$
for semi-metrics.

Let $\xi=\T\sigma$ and $\theta=\T\tau$ for pretopologies $\sigma$
and $\tau$ of countable character. In view of Lemma \ref{lem:loccompactrel},
$\xi$ is locally countably $\Epi_{\I_{1}}\sigma$-compact. Therefore
Theorem \ref{thm:frommech} applies to the effect that 
\[
\xi\times\theta=\T\sigma\times\T\tau\geq\T(\sigma\times\tau)
\]
and $\sigma\times\tau$ is a pretopology of countable character finer
than $\xi\times\theta$. Hence $\xi\times\theta$ has countable $\S_{0}$-character. 
\end{proof}
In fact, the converse of each item in Theorem \ref{thm:Tanakaprod}
is also true \cite{Tanakaprod}, as could be obtained from a slight
variation of Theorem \ref{thm:frommech} in which the inequality only
need to be tested for pretopologies of countable character. As this
was not shown or formulated this way in \cite{DM.products}, we omit
the converse.

\appendix

\appendix

\section{Convergence spaces, notation and conventions}

\subsection{Set-theoretic conventions}

If $X$ is a set, we denote by $2^{X}$ its powerset, by $[X]^{<\infty}$
the set of finite subsets of $X$ and by $[X]^{\omega}$ the set of
countable subsets of $X$. If $\A\subset2^{X}$, we write 
\begin{eqnarray*}
\A^{\uparrow} & := & \left\{ B\subset X:\exists A\in\A,A\subset B\right\} \\
\A^{\cap} & := & \left\{ \bigcap_{S\in\F}S:\F\in[\A]^{<\infty}\right\} \\
\A^{\#} & := & \left\{ B\subset X:\forall A\in\A,A\cap B\neq\emptyset\right\} .
\end{eqnarray*}
A family $\F$ of \emph{non-empty }subsets of $X$ is called a \emph{filter
}if $\F=\F^{\cap}=\F^{\uparrow}$. We denote by $\mathbb{F}X$ the
set of filters on $X$. Note that $2^{X}$ is the only family $\A$
satisfying $\A=\A^{\uparrow}=\A^{\cap}$ that has an empty element.
Thus we sometimes call $\{\emptyset\}^{\uparrow}=2^{X}$ the \emph{degenerate
filter on }$X$. The set $\mathbb{F}X$ is ordered by $\F\leq\G$
if for every $F\in\F$ there is $G\in\G$ with $G\subset F$. Maximal
elements of $\mathbb{F}X$ are called \emph{ultrafilters }and $\F\in\mathbb{F}X$
is an ultrafilter if and only if $\F=\F^{\#}$. We denote by $\mathbb{U}X$
the set of ultrafilters on $X$.

A sequence $\{x_{n}\}_{n=1}^{\infty}$ on a set $X$ induces a filter
\[
(x_{n})_{n}:=\left\{ \left\{ x_{n}:n\geq k\right\} :k\in\omega\right\} ^{\uparrow}.
\]
A filter that is induced by some sequence is called \emph{sequential
filter}. We denote by $\mathbb{E}X$ the set of sequential filters
on $X$. A filter $\F\in\mathbb{F}X$ is called \emph{Fréchet} if
\begin{equation}
\F=\bigwedge_{\mathcal{E}\in\mathbb{E}X,\F\leq\mathcal{E}}\mathcal{E}.\label{eq:FrechetFilter}
\end{equation}
Of course, every \emph{countably based filter}, that is, a filter
with a countable filter-base, is in particular a Fréchet filter.

\subsection{Convergence}

A \emph{convergence }$\xi$ on a set $X$ is a relation between $\mathbb{F}X$
and $X$, denoted 
\[
x\in\lm_{\xi}\F
\]
whenever $(x,\F)\in\xi$ (and we then say that $\F$ \emph{converges
to }$x$ \emph{for }$\xi$), satisfying the following two conditions:
\[
\F\leq\G\then\lm_{\xi}\F\subset\lm_{\xi}\G
\]
\[
x\in\lm_{\xi}\{x\}^{\uparrow},
\]
for every filters $\F$ and $\G$ on $X$, and every $x\in X$.

The pair $(X,\xi)$ is then called a \emph{convergence space}. We
denote by $\lim_{\xi}^{-}(x)$ the set of filters that converge to
$x$ for $\xi$.

Every topology can be seen as a convergence. Indeed, if $\tau$ is
a topology and $\N_{\tau}(x)$ denotes the neighborhood filter of
$x$ for $\tau$, then 
\[
x\in\lm_{\tau}\F\iff\F\geq\N_{\tau}(x)
\]
defines a convergence that completely characterizes $\tau$. Hence,
we do not distinguish between a topology $\tau$ and the convergence
it induces.A convergence is called \emph{Hausdorff }if the cardinality
of $\lim\F$ is at most one, for every filter $\F$. Of course, a
topology is Hausdorff in the usual topological sense if and only if
it is in the convergence sense. A point $x$ of a convergence space
$(X,\xi)$ is \emph{isolated }if $\lm_{\xi}^{-}(x)=\{\{x\}^{\uparrow}\}$.
A \emph{prime }convergence is a convergence with at most one non-isolated
point.

Given two convergences $\xi$ and $\theta$ on the same set $X$,
we say that $\xi$ is \emph{finer than $\theta$ }or that $\theta$
is \emph{coarser than }$\xi$, in symbols $\xi\geq\theta$, if $\lim_{\xi}\F\subset\lim_{\theta}\F$
for every $\F\in\mathbb{F}X$. With this order, the set $\C(X)$ of
convergences on $X$ is a complete lattice whose greatest element
is the discrete topology and least element is the antidiscrete topology,
and for which, given $\Xi\subset\C(X)$,
\[
\lm_{\bigvee\Xi}\F=\bigcap_{\xi\in\Xi}\lm_{\xi}\F\text{ and }\lm_{\bigwedge\Xi}\F=\bigcup_{\xi\in\Xi}\lm_{\xi}\F.
\]

A \emph{base }for a convergence space $(X,\xi)$ is a family $\B$
of subsets of $X$ such that for every $x\in X$ and every $\F$ with
$x\in\lim_{\xi}\F$, there is a filter $\G$ with a filter-base composed
of elements of $\B$ with $x\in\lim_{\xi}\G$ and $\F\leq\G$. If
this property is satisfied only for a specific $x$, then $\B$ is
a \emph{local base at $x$. }

\subsection{Continuity, initial and final constructions}

A map $f$ between two convergence spaces $(X,\xi)$ and $(Y,\sigma)$
is continuous if for every $\F\in\mathbb{F}X$ and $x\in X$,
\[
x\in\lm_{\xi}\F\then f(x)\in\lm_{\sigma}f[\F],
\]
where 
\[
f[\F]:=\{B\subset Y:f^{-1}(B)\in\F\}=\{f(F):F\in\F\}^{\uparrow}.
\]

Consistently with \cite{DM.book}, we denote by $|\xi|$ the underlying
set of a convergence $\xi$, and, if $(X,\xi)$ and $(Y,\sigma)$
are two convergence spaces we often write $f:|\xi|\to|\sigma|$ instead
of $f:X\to Y$ even though one may see it as improper since many different
convergences have the same underlying set. This allows to talk about
the continuity of $f:|\xi|\to|\sigma|$ without having to repeat for
what structure. 

Given a map $f:|\xi|\to Y$, there is the finest convergence $f\xi$
on $Y$ making $f$ continuous (from $\xi$), and given $f:X\to|\sigma|$,
there is the coarsest convergence $f^{-}\sigma$ on $X$ making $f$
continuous (to $\sigma$). The convergences $f\xi$ and $f^{-}\sigma$
are called \emph{final convergence for $f$ and $\xi$ }and \emph{initial
convergence for $f$ and $\sigma$ }respectively. Note that 
\begin{equation}
f:|\xi|\to|\sigma|\text{ is continuous }\iff\xi\geq f^{-}\sigma\iff f\xi\geq\sigma.\label{eq:continuity}
\end{equation}

If $A\subset|\xi|$, the \emph{induced convergence by $\xi$ on $A$,
}or \emph{subspace convergence},\emph{ }is $i^{-}\xi$, where $i:A\to|\xi|$
is the inclusion map. If $\xi$ and $\tau$ are two convergences,
the \emph{product convergence $\xi\times\tau$ }on $|\xi|\times|\tau|$
is the coarsest convergence on $|\xi|\times|\tau|$ making both projections
continuous, that is,
\[
\xi\times\tau:=p_{\xi}^{-}\xi\vee p_{\tau}^{-}\tau,
\]
where $p_{\xi}:|\xi|\times|\tau|\to|\xi|$ and $p_{\tau}:|\xi|\times|\tau|\to|\tau|$
are the projections defined by $p_{\xi}(x,y)=x$ and $p_{\tau}(x,y)=y$
respectively.

\subsection{Topologies and pretopologies}

In fact, the category $\mathbf{Top}$ of topological spaces and continuous
maps is a reflective subcategory of $\mathbf{Conv}$ and the corresponding
reflector $\T$, called \emph{topologizer}, associates to each convergence
$\xi$ on $X$ its \emph{topological modification $\T\xi$, }which
is the finest topology on $X$ among those coarser than $\xi$ (in
$\mathbf{Conv}$). Concretely, $\T\xi$ is the topology whose closed
sets are the subsets of $|\xi|$ that are $\xi$-\emph{closed, }that
is, subsets $C$ satisfying
\[
C\in\F^{\#}\then\lm_{\xi}\F\subset C.
\]
 A subset $O$ of $|\xi|$ is $\xi$-open if its complement is closed,
equivalently if 
\[
\lm_{\xi}\F\cap O\neq\emptyset\then O\in\F.
\]

Because $\mathbf{Top}$ is a (full) reflective subcategory of $\mathbf{Conv}$,
$\mathbf{Top}$ is closed under initial constructions so that a subspace
of a topological convergence space is topological and a product of
topological convergence spaces is topological. However, the convergence
induced by $\T\xi$ on a subset $A$ and $\T(i^{-}\xi)$ do not need
to coincide, and similarly, the topological modification of a product
generally does not coincide with the product of the topological modifications.

Given a convergence $\xi$ and $x\in|\xi|$, the filter 
\[
\V_{\xi}(x):=\bigwedge_{\F\in\lm_{\xi}^{-}(x)}\F
\]
is called the \emph{vicinity filter of} $x$ for $\xi$. In general,
$\V_{\xi}(x)$ does not need to converge to $x$ for $\xi$. If $x\in\lim_{\xi}\V_{\xi}(x)$
for all $x\in|\xi|$, we say that $\xi$ is a \emph{pretopology }or
is \emph{pretopological}. Note that a local base at $x$ of a pretopology
is necessarily a filter-base of $\V_{\xi}(x)$.

The category $\mathbf{PrTop}$ of pretopological spaces and continuous
maps is a (full) reflective subcategory of $\mathbf{Conv}$. The corresponding
reflector $\S_{0}$, called \emph{pretopologizer}, associates to each
convergence $\xi$ its \emph{pretopological modification }$\S_{0}\xi$,
which is the finest among the pretopologies coarser than $\xi$. Explicitly,
$x\in\lim_{\S_{0}\xi}\F$ if $\F\geq\V_{\xi}(x)$ so that $\V_{\xi}(x)=\V_{\S_{0}\xi}(x)$.
Topologies are in particular pretopological, so that $\S_{0}\geq\T$.
Moreover, 
\[
O\text{ is }\xi\text{-open }\iff O\in\bigcap_{x\in O}\V_{\xi}(x),
\]
and $\xi$ is topological if and only if $\V_{\xi}(x)$ has a filter-base
composed of open sets, in which case $\V_{\xi}(x)=\N_{\xi}(x)$. We
distinguish between the \emph{inherence 
\begin{equation}
\inh_{\xi}A=\{x\in X:A\in\V_{\xi}(x)\}\label{eq:inh}
\end{equation}
of $A\subset|\xi|$ }and its \emph{interior 
\[
\intr_{\xi}A=\{x\in X:A\in\N_{\xi}(x)\}.
\]
}

In contrast to $\T$, the reflector $\S_{0}$ does commute with initial
convergence (so that the pretopological modification of an induced
convergence is the convergence induced by the pretopological modification),
but not with suprema, hence not with products. 

\subsection{Sequential, Fréchet, and strongly Fréchet convergences}

To a convergence $\xi$, we associate its \emph{sequentially based
modification }$\Seq\xi$ defined by 
\[
\lm_{\Seq\xi}\F=\bigcup_{\mathcal{E}\in\mathbb{E}|\xi|,\mathcal{E}\leq\F}\lm_{\xi}\mathcal{E}.
\]

A useful immediate consequence of \cite[Corollary 10]{DM.seq} (that
can be traced all the way back to \cite{KAN}, but \cite{DM.seq}
gives a formulation more consistent with our terminology) is:
\begin{prop}
\label{prop:coincideSeq} If $\sigma$ is a Hausdorff pretopology,
then 
\[
\Seq\sigma=\Seq\T\sigma.
\]
\end{prop}
In the same vein, letting $\mathbb{F}_{1}X$ denote the set of filters
with a countable filter-base, we can associate to each convergence
$\xi$ its\emph{ modification of countable character }$\I_{1}\xi$
defined by 
\[
\lm_{\I_{1}\xi}\F=\bigcup_{\H\in\mathbb{F}_{1}|\xi|,\H\leq\F}\lm_{\xi}\H.
\]

Both $\Seq$ and $\I_{1}$ are concrete coreflectors from $\mathbf{Conv}$
to the full categories of sequentially based and convergences of countable
character respectively.

Recall that a topological space is \emph{sequential }if every sequentially
closed subset (that is, subset that contains the limit point of every
sequence on it) is closed. It is easily seen that a subset of a topology
$\xi$ is sequentially closed if and only if it is $\Seq\xi$-closed,
so that, $\xi$ is sequential if and only if $\xi=\T\Seq\xi$. As
$\sigma\leq\Seq\sigma$ for every convergence $\sigma$, the inequality
$\xi\leq\T\Seq\xi$ is true for every topology $\xi$. Hence if $\xi=\T\xi$
then
\[
\xi\text{ is sequential}\iff\xi\geq\T\Seq\xi,
\]
 and it turns out that 
\begin{equation}
\xi\geq\T\Seq\xi\iff\xi\geq\T\I_{1}\xi\label{eq:sequential}
\end{equation}
and we take one or the other of these inequalities as a definition
of a \emph{sequential convergence.}

A topological space $X$ is \emph{Fréchet }if for every $x\in X$
and $A\subset X$, if $x\in\cl A$ then there is a sequence on $A$
converging to $x$. It is easily seen (e.g., \cite{KR.decompositionseries}\cite{quest2})
that if $\xi$ is a topology then
\begin{equation}
\xi\text{ Fréchet }\iff\xi\geq\S_{0}\Seq\xi\iff\xi\geq\S_{0}\I_{1}\xi,\label{eq:Frechet}
\end{equation}
and we take either one of these inequalities as a definition of a
\emph{Fréchet convergence. }Note that a pretopology is Fréchet if
and only if each vicinity filter is a Fréchet filter in the sense
of (\ref{eq:FrechetFilter}).

The \emph{paratopologizer} $\S_{1}$ is defined by 
\[
\lm_{\S_{1}\xi}\F=\bigcap_{\mathbb{F}_{1}\ni\H\#\F}\adh_{\xi}\H
\]
 and defines a concrete reflector. A fixed point of $\S_{1}$ is called
a \emph{paratopology.}

A topological space $X$ is \emph{strongly Fréchet }if whenever $x\in\adh\H$
for a countably based filter $\H$ on $X$, there is a sequential
filter $(x_{n})_{n}$ with $(x_{n})_{n}\geq\H$ and $x\in\lim(x_{n})_{n}$.
It is easily seen (e.g., \cite{quest2}) that if $\xi$ is a topology
then
\begin{equation}
\xi\text{ strongly Fréchet }\iff\xi\geq\S_{1}\Seq\xi\iff\xi\geq\S_{1}\I_{1}\xi,\label{eq:Frechet-1}
\end{equation}
and we take either one of these inequalities as a definition of a
\emph{strongly Fréchet convergence.}

\subsection{Properties of concrete functors}

Several other topological properties can be characterized with the
aid of a functorial inequality; see \cite{quest2,DM.book}. All functors
considered here (in particular $\T$, $\S_{0}$, $\S_{1}$, $\Seq$
and $\I_{1}$ ) are concrete endofunctors of $\mathbf{Conv}$, and
as such each satisfy the following properties of a modifier $F$ acting
on convergence spaces (for all $\xi$,$\theta$ and $f$):
\[
|F\xi|=F|\xi|
\]
\[
\xi\leq\theta\then F\xi\leq F\theta
\]
\[
F(f^{-}\theta)\geq f^{-}(F\theta)
\]

All five functors $\T$, $\S_{0}$, $\S_{1}$, $\Seq$ and $\I_{1}$
are also idempotent, that is, satisfy $F(F\xi)=F\xi$ for all $\xi$.
The reflectors $\T$, $\S_{0}$ and $\S_{1}$ are additionally contractive
($F\xi\leq\xi$ for all $\xi$) and the coreflectors $\Seq$ and $\I_{1}$
are additionally expansive ($F\xi\geq\xi$ for all $\xi$).

\subsection{Compactness and local compactness}

A subset $K$ of a convergence space is \emph{$\xi$-compact }if $\lim_{\xi}\U\cap K\neq\emptyset$
for every ultrafilter $\U$ on $K$, and \emph{countably compact }if
every countably based filter $\H$ with $K\in\H^{\#}$, there is an
ultrafilter $\U\geq\H$ with $\lim_{\xi}\U\cap K\neq\emptyset$. Given
two convergences $\xi$ and $\sigma$ on the same set, we say that
$\xi$ is \emph{locally (countably) $\sigma$-compact }if every $\xi$-convergent
filter has a (countably) $\sigma$-compact element.

\bibliographystyle{plain}

\begin{thebibliography}{10}

\bibitem{arh.mapppings}
A.~V. Arhangel'skii.
\newblock Mappings and spaces.
\newblock {\em Russian Math. Surveys}, {\bf 21}:115--162, 1966.

\bibitem{quest2}
S.~Dolecki.
\newblock Convergence-theoretic methods in quotient quest.
\newblock {\em Topology Appl.}, {\bf 73}:1--21, 1996.

\bibitem{DG}
S.~Dolecki and G.~H. Greco.
\newblock Topologically maximal pretopologies.
\newblock {\em Studia Math.}, {\bf 77}:265--281, 1984.

\bibitem{DM.seq}
S.~Dolecki, F.~Jordan, and F.~Mynard.
\newblock Reflective classes of sequentially based convergence spaces,
  sequential continuity and sequence-rich filters.
\newblock {\em Topology Proceedings}, 31(2):457--479, 2007.

\bibitem{coreflrevisited}
S.~Dolecki and F.~Mynard.
\newblock Productively sequential spaces.
\newblock {\em to appear in Math. Slovaca}.

\bibitem{DM.products}
S.~Dolecki and F.~Mynard.
\newblock Convergence-theoretic mechanisms behind product theorems.
\newblock {\em Topology and its Applications}, {\bf 104}:67--99, 2000.

\bibitem{DM.book}
S.~Dolecki and F.~Mynard.
\newblock {\em Convergence Foundations of Topology}.
\newblock World Scientific, 2016.

\bibitem{dolecki1998topologically}
S.~Dolecki and M.~Pillot.
\newblock Topologically maximal convergences, accessibility, and covering maps.
\newblock {\em Mathematica Bohemica}, 123(4):371--384, 1998.

\bibitem{Gruenhage84}
G.~Gruenhage.
\newblock {\em Generalized Metric Spaces}, volume Handbook of Set-Theoretic
  Topology, pages 423--501.
\newblock Elsevier, {K}. {K}unen and {J}. {E}. {V}aughan edition, 1984.

\bibitem{harley1976symmetrizable}
P.~W. Harley~III and R.M. Stephenson~Jr.
\newblock Symmetrizable and related spaces.
\newblock {\em Transactions of the American Mathematical Society}, pages
  89--111, 1976.

\bibitem{heath1962arc}
R.~Heath.
\newblock Arc-wise connectedness in semi-metric spaces.
\newblock {\em Pacific J. of Math.}, 12(4):1301--1319, 1962.

\bibitem{hong1999notes}
W.~C. Hong.
\newblock Notes on {F}r{\'e}chet spaces.
\newblock {\em International Journal of Mathematics and Mathematical Sciences},
  22(3):659--665, 1999.

\bibitem{hong2002note}
W.~C. Hong.
\newblock A note on weakly first countable spaces.
\newblock {\em Commuications-Korean Mathematical Society}, 17(3):531--534,
  2002.

\bibitem{KAN}
L.~V. Kantorovich, B.~Z. Vulich, and A.~G. Pinsker.
\newblock {\em Functional Analysis in Vector Order Spaces}.
\newblock GITTL, 1950.

\bibitem{KR.decompositionseries}
D.~C. Kent and G.~D. Richardson.
\newblock The decomposition series of a convergence space.
\newblock {\em Czechoslovak Math. J.}, 23(98):437--446, 1973.

\bibitem{quest}
E.~Michael.
\newblock A quintuple quotient quest.
\newblock {\em Gen. Topology Appl.}, {\bf 2}:91--138, 1972.

\bibitem{mynard.strong}
F.~Mynard.
\newblock Strongly sequential spaces.
\newblock {\em Comment. Math. Univ. Carolinae}, {\bf 41}:143--153, 2000.

\bibitem{mynard}
F.~Mynard.
\newblock Coreflectively modified continuous duality applied to classical
  product theorems.
\newblock {\em Applied General Topology}, 2 (2):119--154, 2002.

\bibitem{Myn.strongmore}
F.~Mynard.
\newblock More on strongly sequential spaces.
\newblock {\em Comment. Math. Univ. Carolinae}, 43(3):525--530, 2002.

\bibitem{Mynard.survey}
F.~Mynard.
\newblock Coreflectively modified duality.
\newblock {\em Rocky Mountain J. of Math.}, 34(2):733--758, 2004.

\bibitem{MR0367935}
S.~{\u{I}}. Nedev.
\newblock {$o$}-metrizable spaces.
\newblock {\em Trudy Moskov. Mat. Ob\v s\v c.}, 24:201--236, 1971.

\bibitem{MR0350706}
F.~Siwiec.
\newblock On defining a space by a weak base.
\newblock {\em Pacific J. Math.}, 52:233--245, 1974.

\bibitem{Tanakonsym}
Y.~Tanaka.
\newblock On symmetric spaces.
\newblock {\em Proc. Japan Acad}, 49:106--111, 1973.

\bibitem{Tanakaprod}
Y.~Tanaka.
\newblock Note on products of symmetric spaces.
\newblock {\em Proc. Japan Acad.}, 50:152--154, 1974.

\bibitem{why}
G.~T. Whyburn.
\newblock Accessibility spaces.
\newblock {\em Proc. Amer. Math. Soc.}, {\bf 24}:181--185, 1970.

\end{thebibliography}

\end{document}